\documentclass[12pt, a4paper]{amsart}
\usepackage[latin1]{inputenc}
\usepackage{amsfonts,amsmath,a4wide}

\newtheorem{thm}{Theorem}
\newtheorem{prop}{Proposition}

\theoremstyle{definition}
\newtheorem{dfn}{Definition}
\newtheorem{rmk}{Remark}
\newtheorem{example}{Example}

\begin{document}
\title[Betti numbers of graded modules]{Betti numbers of graded modules and the {M}ultiplicity {C}onjecture in the non-{C}ohen-{M}acaulay case}
\date{\today \qquad(Preliminary version)}
\author{Mats Boij} \address{Department of Mathematics, KTH \\ S--100
  44 Stockholm \\ Sweden} \email{boij@kth.se}

\author{Jonas Söderberg} \address{Department of Mathematics, KTH \\
  S--100 44 Stockholm \\ Sweden} \email{jonasso@math.kth.se}


\begin{abstract} 
We use the results by Eisenbud and Schreyer~\cite{Eisenbud-Schreyer} to prove that any Betti diagram of a graded module over a standard graded polynomial ring is a positive linear combination Betti diagrams of modules with a pure resolution. This implies the Multiplicity Conjecture of Herzog, Huneke and Srinivasan~\cite{Herzog-Srinivasan} for modules that are not necessarily Cohen-Macaulay. We give a combinatorial proof of the convexity of the simplicial fan spanned by the pure diagrams. 
\end{abstract}

\maketitle

\section{Introduction}
The formula for the multiplicity of a standard graded algebra with a pure resolution found by C.~Huneke and M.~Miller~\cite{Huneke-Miller}, led to the formulation of the Multiplicity Conjecture by C.~Huneke and H.~Srinivasan, which was later generalized and published by J.~Herzog and H.~Srinivasan~\cite{Herzog-Srinivasan}. 

In a series of papers\footnote{Citations will be added.} these versions of the Multiplicity conjecture have been proven in many cases. Recently D.~Eisenbud and F.-O.~Schreyer proved the conjecture in the Cohen-Macaulay case by proving a set of conjectures formulated by the authors~\cite{Boij-Soderberg} on the set of possible Betti diagrams up to multiplication by positive rational numbers.  

In the work of D.~Eisenbud and F.-O.~Schreyer, they introduce a set of linear functionals defined on the space of possible Betti diagrams. The linear functionals are given by certain cohomology tables of vector bundles on $\mathbb P^{n-1}$ and they show that the supporting hyperplanes of the exterior facets of the simplicial fan given by the pure Betti diagrams are given by the vanishing of these linear functionals, while the functionals are non-negative on the Betti diagram of any minimal free resolution. 

In this paper we generalize the contruction given in our previous paper in order to include Cohen-Macaulay pure Betti diagrams of various codimensions. We show that the linear functionals, similar to the ones introduced  by D.~Eisenbud and F.-O.~Schreyer define the supporting hyperplanes of the simplicial fan. Furthermore, these new linear functionals are limits of the functionals given by D.~Eisenbud and F.-O.~Schreyer which allows us to conclude that all Betti diagrams of graded modules can be uniquely written as a positive linear combination of pure diagrams in a totally ordered chain. Together with the existence of modules with pure resolutions, proved by D.~Eisenbud, G.~Fl{\o}ystad and J.~Weyman~\cite{Eisenbud-Floystad-Weyman} in characteristic zero, and by D.~Eisenbud and F.-O.~Schreyer~\cite{Eisenbud-Schreyer} in general, this gives a complete classificationo of the possible Betti diagrams up to multiplication by scalars. 

As a consequence, we get the Multiplicity Conjecture for algebras and modules that are not necessarily Cohen-Macaulay. In fact, we get a stronger version of the inequalities of the Multiplicity Conjecture in terms of the Hilbert series of the module which is bounded from below by the Hilbert series corresponding to the lowest shifts in a minimal free resolution while it is bounded from above by the Hilbert series corresponding to the highest shifts in the first $s+1$ terms of the resolution, where $s$ is the codimension of the module. 

We also give a combinatorial proof of the convexity of the simplicial fan spanned by the pure diagrams, even though this convexity is an implicit consequence of the results involving the linear functionals. 

Furthermore, we show that if we choose the basis of pure diagrams in a certain way, all the coefficients in the expansion of an actual Betti diagram into a chain of pure diagrams are non-negative integers.  

\section{The partially ordered set of pure Betti diagrams}
Let $R=k[x_1,x_2,\dots,x_n]$ be the polynomial ring with the standard grading. For any finitely generated graded module $M$, we have a minimal free resolution
$$
0\longrightarrow F_n \longrightarrow F_{n-1}\longrightarrow \cdots \longrightarrow F_1\longrightarrow F_0\longrightarrow M\longrightarrow 0
$$
where 
$$
F_i = \bigoplus_{j\in \mathbb Z} R(-j)^{\beta_{i,j}}, \qquad i=0,1,\dots,n,
$$
and we get that the Hilbert series of $M$ can be recovered from the Betti  numbers, $\beta_{i,j}$, by
$$
H(M,t) = \frac{1}{(1-t)^n} \sum_{i=0}^n \sum_{j\in \mathbb Z} (-1)^i \beta_{i,j}t^j.
$$
It was noted by J.~Herzog and M.~K\"uhl~\cite{Herzog-Kuhl} that this gives us $s$ linearly independent equations 
$$
\sum_{i=0}^n \sum_{j\in \mathbb Z} (-1)^i \beta_{i,j} j^m=0, \qquad m=0,1,\dots,s-1
$$
where $s$ is the codimension of $M$. 

Furthermore, they proved that in the case when $M$ is Cohen-Macaulay and the resolution is \emph{pure}, i.e., if $F_i=R(-d_i)$ for some integers $d_0,d_1,\dots,d_s$, we get a unique solution to these equations given by 
$$
\beta_{i,j} = \left\{\begin{array}{rl}
\displaystyle(-1)^i\prod_{\underset{j\ne i}{j=0}}^s\frac{1}{d_j-d_i},&\qquad j=d_i\\
0,&\qquad j\ne d_i
\end{array}
\right.,
$$
for $i=0,1,\dots,s$.

\begin{dfn} 
For an increasing sequence of integers $\underline d=(d_0,d_1,\dots,d_s)$, where $0\le s\le n$, we denote by $\pi(\underline d)$ the matrix in $\mathbb Q^{n-1}\times \mathbb Q^{\mathbb Z}$ given by 
$$
\pi(\underline d)_{i,j}=
(-1)^i\prod_{\underset{j\ne i}{j=0}}^s\frac{1}{d_j-d_i}
,\qquad\hbox{for $j=d_i$}
$$
and zero elsewhere. We will call this the \emph{pure diagram} given by the degree sequence $\underline d=(d_0,d_2,\dots,d_s)$. We will use the notation $d_i(\pi)$ to denote the degree, $d_i$, when $\pi=\pi(\underline d)=\pi(d_0,d_1,\dots,d_s)$.

For degree sequences with $d_0=0$, we will also use the \emph{normalized pure diagram} 
$$
\bar\pi(\underline d) = d_1d_2\cdots d_s\pi(\underline d)
$$
so that normalized pure diagrams have $\bar\pi_{0,0}=1$. 
\end{dfn}

We define a partial ordering on the set of pure diagrams, extending the ordering used for Cohen-Macaulay diagrams of a fixed codimension. 

\begin{dfn} We say that $\pi(d_0,d_1,\dots,d_s)\le \pi (d_0',d_1',\dots,d_t')$
if the $s\ge t$ and $d_i\le d_i'$ for $i=0,1,\dots,t$. 
\end{dfn}

As in the case of Cohen-Macaulay pure diagrams, we get a simplicial structure given by the maximal chains of pure diagrams, which spans simplicial cones. In order to have maximal chains in this setting, we have to fix a bound on the region we are considering. We can do this by restricting the degrees to be in a given region $M+i\le d_i\le N+i$. We will denote the subspace generated by Betti diagrams with these restrictions by $B_{M,M}=\mathbb Q^{n+1}\times \mathbb Q^{N-M+1}$. Furthermore, we denote by $B_{M,N}^s$ the subspace of $B_{M,N}$ of diagrams satisfying the $s$ first Herzog-K\"uhl equations. 

\begin{prop} \label{prop:basis}
  For $s=0,1,\dots,n$, we have that any maximal chain of pure diagrams of codimension at least $s$ in $B_{M,N}$ form a basis for $B_{M,N}^s$.
\end{prop}

\begin{proof}
  For any interval of length one $\pi<\pi'$ there is a unique non-zero entry in $\pi$ which is zero in all pure diagrams above $\pi'$. Thus the pure diagrams in any maximal chain are linearly independent. The number of elements in any maximal chain of pure diagrams of codimension at least $s$ is $(n+1)(N-M)+n-s+1$, since we have $n+1$ positions that has to be raised $N-M$ steps and then $n-s+1$ times when the codimension is lowered by one. On the other hand, we have that the dimension of $B_{M,N}^s$ is $(n+1)(N-M+1)-s$, since we have $s$ independent equations on $B_{M,N}$. 
\end{proof}

By looking at the order in which the different positions of a Betti diagram in $B_{M,N}$ disappear when going along a maximal chain of pure diagrams we get the following observation: 

\begin{prop} \label{prop:tableaux}
The maximal chains of pure diagrams in $B_{M,N}$ are in one to one correspondence with numberings of the entries 
of an $(N-M+1)\times (n+1)$-matrix which are increasing to the left and downwards.
\end{prop}

\begin{example}
For $n=2$, $M=0$ and $N=1$, the numbering 
$$
\begin{array}{|c|c|c|}\hline
4&3&1\\ \hline
6&5&2\\ \hline
\end{array}
$$
corresponds to the maximal chain
$$
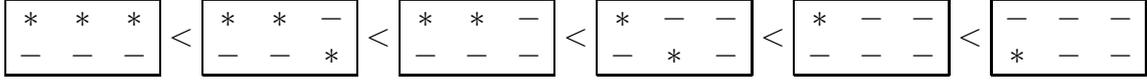
\begin{array}{|ccc|}\hline 
\ast&\ast&\ast\\ -&-&-\\ \hline 
\end{array}
<
\begin{array}{|ccc|}\hline 
\ast&\ast&-\\ -&-&\ast\\ \hline 
\end{array}
<
\begin{array}{|ccc|}\hline 
\ast&\ast&-\\ -&-&-\\ \hline 
\end{array}
<
\begin{array}{|ccc|}\hline 
\ast&-&-\\ -&\ast&-\\ \hline 
\end{array}
<
\begin{array}{|ccc|}\hline 
\ast&-&-\\ -&-&-\\ \hline 
\end{array}
<
\begin{array}{|ccc|}\hline 
-&-&-\\ \ast&-&-\\ \hline 
\end{array}
$$
and there are four other maximal chains corresponding to 
$$
\begin{array}{|c|c|c|}\hline
3&2&1\\ \hline
6&5&4\\ \hline
\end{array},\quad
\begin{array}{|c|c|c|}\hline
4&2&1\\ \hline
6&5&3\\ \hline
\end{array},\quad
\begin{array}{|c|c|c|}\hline
5&2&1\\ \hline
6&4&3\\ \hline
\end{array},\quad\hbox{and}\quad
\begin{array}{|c|c|c|}\hline
5&3&1\\ \hline
6&4&2\\ \hline
\end{array}.
$$
\end{example}

\section{Description of the boundary facets}

We know that the partially ordered set of pure diagrams in $B_{M,N}$ give rise to a simplicial fan $\Delta$, where the faces are the totally ordered subsets. The facets of this simplicial complex correspond to maximal chains in the partially ordered set. According to Proposition~\ref{prop:basis} each such set is a basis for the space $B_{M,N}$. If we look at normalized Betti diagrams, we get a simplicial complex which is the hyperplane section of the simplicial fan. This is the complex described in our previous paper, in the case of Cohen-Macaulay diagrams.  

\begin{prop}\label{prop:simlicialfan}
  The simplicial cones spanned by the totally ordered sets of pure diagrams in $B_{M,N}$ form a simplicial fan. 
\end{prop}

\begin{proof}
  We need to show that the cones meet only along faces. This is the same as to say that any element which can be written as a positive linear combination of pure diagrams in a chain can be written so in a unique way. If such a sum has only one term, the uniqueness is trivial. Thus, suppose that a diagram in $B_{M,N}$ can be written as a positive linear combination of totally ordered pure diagrams in two different ways and that this is the minimal number of terms in such an example. We then have 
$$
\beta=\sum_{i=1}^m \lambda_i \pi_i = \sum_{j=1}^k \mu_j \pi_j'
$$
where all the coefficients are positive. Look at the lowest degree in which $\beta$ is non-zero for each column. These degrees have to be given by the degrees in $\pi_1$ and by the degrees in $\pi_1'$. Thus we must have $\pi_1=\pi_1'$. If $\lambda>\mu$, we can subtract $\mu\pi_1$ from $\beta$ and from both sums to get 
$$
\beta-\lambda_1\pi_1=\sum_{i=2}^m \lambda_i \pi_i =(\mu_1-\lambda_1)\pi_1 \sum_{j=2}^k \mu_j \pi_j'  
$$
Since all the pure diagrams in the left expression are greater than $\pi_1$, the degrees in which $\beta-\lambda_1\pi_1$ is non-zero have to be given by $\pi_2$, but then the coefficient of $\pi_1$ in the right hand expression has to be zero, i.e., $\lambda_1=\mu_1$. Since this was assumed to be a minimal example, we get that the expressions for $\beta-\lambda_1\pi_1$ are term wise equal and so were the original expressions for the diagram $\beta$. 
\end{proof}

We will now show that the coefficients of the pure diagrams when a Betti diagram is expanded in the basis given by a maximal chain have nice expressions. In particular, we see that the coefficients are integers. Moreover, it will give us expressions for the inequalities that define the simplicial fan similar to the inequalities used by D.~Eisenbud and F.-O.~Schreyer to prove our conjectures in the Cohen-Macaulay case. 

\begin{prop}\label{prop:coefficients}
The coefficient of $\pi_1=\pi(d_0,d_1,\dots,d_m)$ when a Betti diagram $\beta$ in $B_{M,N}^s$ is expanded in a basis containing $\pi_0<\pi_1<\pi_2$ is given by
$$
\sum_{i=0}^n \sum_{d=M}^{d_i(\pi_0)} (-1)^i \prod_{\underset{j\ne k}{j=0}}^{m}(d_j-d)\beta_{i,d}, 
$$
when $\pi_1$ differs from $\pi_0$ in codimension and from $\pi_2$ in column $k$,  by 
$$
\sum_{i=0}^n \sum_{d=M}^{d_i(\pi_0)} (-1)^i(d_k-d_m) \prod_{\underset{j\ne k}{j=0}}^{m-1}(d_j-d)\beta_{i,d}, \qquad 
$$
when $\pi_1$ differs from $\pi_2$ in codimension and from $\pi_0$ in column $k$, by 
$$
\sum_{i=0}^n \sum_{d=M}^{d_i(\pi_0)} (-1)^i(d_\ell-d_k) \prod_{\underset{j\notin\{k,\ell\}}{j=0}}^{m}(d_j-d)\beta_{i,d}, \qquad 
$$
when $\pi_1$ differs from $\pi_2$ in column $k$ and from $\pi_0$ in column $\ell\ne k$, and by 
$$
\sum_{i=0}^n \sum_{d=M}^{d_i(\pi_0)} (-1)^i \prod_{j=0}^{m-1}(d_j-d)\beta_{i,d}, \qquad 
$$
when $\pi_1$ differs from $\pi_0$ and $\pi_2$ in codimension. In particular, all the coefficients are integers. 
\end{prop}

\begin{rmk} For simplicity of notation, we use the convention that $d_i(\pi)=N$ if $i$ is greater than the codimension of $\pi$.
\end{rmk}

\begin{proof}
In all four cases, the coefficients of $\beta_{i,d}$ are zero for pure diagrams $\pi\ge \pi_2$ by construction since the sums are taken only up to degree $d_i(\pi_0)$ and the coefficients are zero in all positions where $d_i(\pi_0)=d_i(\pi_2)$. 

Furthermore, the expressions are zero on all pure diagrams $\pi$, where $\pi\le \pi_0$ by the Herzog-K\"uhl equations. Indeed, for such pure diagrams, we can extend the summation over $d$ be over all degrees from $M$ to $N$. When we do this we get and expand the products we get polynomials in $d$ and therefore, the expressions are linear combinations of the Herzog-K\"uhl equations 
$$
\sum_{i=0}^n\sum_{d=M}^N (-1)^i \beta_{i,d} d^j=0,
$$
for $j=0,1,\dots,m-1$ in the second and third case, and for $j=0,1,\dots,m$ in the first and fourth case. Observe that all pure diagrams $\pi$ with $\pi\le \pi_0$, have codimension at least the codimension of $\pi_0$ and hence satisfy these Herzog-K\"uhl equations. 

It remains to show that the value of the expressions are one on the pure diagram $\pi_1$. 
In the first expression, the only non-zero term is in column $k$ and equals  
$$
(-1)^k\prod_{\underset{j\ne k}{j=0}}^m (d_j-d_k) \cdot (-1)^k \prod_{\underset{j\ne k}{j=0}}^m \frac{1}{d_j-d_k}  = 1,
$$
in the second expression, the only non-zero term is in column $m$ and equals
$$
(-1)^m(d_k-d_m)\prod_{\underset{j\ne k}{j=0}}^m (d_j-d_m) \cdot (-1)^m \prod_{j=0}^{m-1} \frac{1}{d_j-d_m} = 1,
$$
in the third expression, the only non-zero term is in column $k$ and equals 
$$
(-1)^k(d_\ell-d_k)\prod_{\underset{j\notin\{k,\ell\}}{j=0}}^m (d_j-d_k) \cdot (-1)^k \prod_{\underset{j\ne \ell}{j=0}}^m \frac{1}{d_j-d_k} = 1
$$
and in the fourth expression, the non-zero term is in column $m$ and equals
$$
(-1)^m \prod_{j=0}^{m-1}(d_j-d_m)\cdot (-1)^{m} \prod_{j=0}^{m-1} \frac{1}{d_j-d_m} =1.
$$
\end{proof}

\begin{example}\label{ex:functionals}
  When $n=3$, $M=0$ and $N=2$ we can choose the basis of $B_{0,2}$ given by the numbering
$$
  \begin{array}{|c|c|c|c|}\hline 10&4&3&1\\ \hline 11&6&5&2\\ \hline 12&9&8&7\\ \hline 
\end{array}
$$
When expanding a Betti diagram 
$$
\begin{array}{c|cccc}
&\beta_0&\beta_1&\beta_2&\beta_3\\ \hline
0:&\beta_{0,0}&\beta_{1,1}&\beta_{2,2}&\beta_{3,3}\\
1:&\beta_{0,1}&\beta_{1,2}&\beta_{2,3}&\beta_{3,4}\\
2:&\beta_{0,2}&\beta_{1,3}&\beta_{2,4}&\beta_{3,5}\\
\end{array}
$$
into this chain we get the coefficients by applying the linear functionals corresponding to the following matrices:
$$
\begin{array}{cccc}
  \begin{array}{|cccc|}\hline
    0&0&0&6\\0&0&0&0\\0&0&0&0\\ \hline 1&&&\\ \hline
    0&-4&6&-6\\0&0&6&4\\0&0&0&0 \\ \hline5&&&\\ \hline
    0&1&-2&3\\0&2&-3&4\\0&3&-4&5 \\ \hline9&&&\\\hline
  \end{array}
  &\begin{array}{|cccc|}\hline
     0&0&0&0\\0&0&0&24\\0&0&0&0\\ \hline 2&&&\\\hline
     0&8&-12&12\\0&12&-12&8\\0&0&0&0 \\ \hline 6&&&\\\hline
     1&-1&1&-1\\0&-1&1&-1\\0&-1&1&-1 \\ \hline10&&&\\\hline
   \end{array}
  &\begin{array}{|cccc|}\hline
     0&0&6&-18\\0&0&0&-36\\0&0&0&0\\ \hline 3&&&\\\hline
     0&-6&8&-6\\0&-8&6&0\\0&0&0&10 \\ \hline 7&&&\\\hline
     0&0&0&0\\1&0&0&0\\0&0&0&0  \\ \hline 11&&&\\\hline
   \end{array}
  &\begin{array}{|cccc|}\hline
     0&8&-12&12\\0&0&0&8\\0&0&0&0\\ \hline4&&&\\\hline
     0&2&-2&0\\0&2&0&-4\\0&0&4&-10 \\ \hline8&&&\\\hline
     0&0&0&0\\0&0&0&0\\1&0&0&0  \\ \hline 12&&&\\\hline
   \end{array}
\end{array}
$$
Matrices $1$, $2$, $11$ and $12$ correspond to individual Betti numbers as we will see in the next proposition. Matrices $3$, $4$, $5$ and $6$ are of the third kind described in Proposition~\ref{prop:coefficients} since they correspond to two consecutive degree changes in distinct columns. Matrix $7$ is of the second kind, since it corresponds to a degree change followed by a change in codimension. Matrices $8$ and $9$ are of the fourth kind, since they correspond to two consecutive changes of codimension. Matrix $10$ is of the first kind since it corresponds to a change in codimension followed by a degree change. 
\end{example}

In order to prove that the simplicial fan is convex and in order to prove that any Betti diagram of a module of codimension $s$ is a positive linear combination of pure diagrams,  we need to know what is the boundary of the simplical fan. The description is very similar to the description given in our previous paper and we only have to add one more kind of boundary facet.

\begin{prop}\label{prop:boundary}
A facet of the boundary of the simplical fan given by the pure diagrams in $B_{M,N}^s$ is given by removing one element from a maximal chain such that there is a unique way to complete it to a maximal chain. There are four different cases: 
\begin{itemize}
\item[i)] The removed element is maximal or minimal, or
\item[ii)] the removed element is in the middle of a chain of three degree sequences which differ in one single column, i.e.,
  $$
  \pi(\underline d)<\pi(\underline d')<\pi(\underline d''),\qquad d_i=d'_i-1=d''_i-2, \quad\hbox{or}
  $$
\item[iii)] the removed element differ from the adjacent vertices in two adjacent degrees, i.e.,  
  $$
  \pi(\underline d)<\pi(\underline d')<\pi(\underline d''), \qquad d_{i+1}=d'_{i+1}-1,\quad d'_i=d''_i-1, \quad \hbox{or}
  $$
\item[iv)] the removed element differ from the adjacent vertices in codimension, i.e.,  
  $$
  \pi(\underline d)<\pi(\underline d')<\pi(\underline d''), \qquad\mathrm{codim}(\pi(\underline d))=\mathrm{codim}(\pi(\underline d''))+2.
  $$
\end{itemize}
\end{prop}

\begin{proof} A boundary facet is a codimension one face of a simplicial cone and hence given by the removing of one vertex. Since it is on the boundary, it is contained in a unique maximal dimensional cone, so there has to be a unique extension of the chain into a maximal chain. 

This clearly happens if we remove the maximal or the minimal element, since there are unique such elements in the partially ordered set. Suppose therefore that we remove $\pi(\underline d')$ and that the adjacent vertices in the chain are $\pi(\underline d)<\pi(\underline d')$ and $\pi(\underline d'')>\pi(\underline d')$. If the difference between $\pi(\underline d)$ and $\pi(\underline d'')$ are in two columns that are not adjacent, we can find another element between $\pi(\underline d)$ and $\pi(\underline d'')$ by exchanging the order in which the two degrees are increased. This can also be done if the columns are adjacent, but the degrees differ by more than one. 

If one of the two differ from $\pi(\underline d')$ in codimension and the other by an increase of the degree in one column, we can alter the order and obtain another element between $\pi(\underline d)$ and $\pi(\underline d'')$. 

The remaining cases are those described by ii), iii) and iv).
\end{proof}

\begin{rmk} In Example~\ref{ex:functionals} matrices $1$ and $12$ correspond to the first kind of facets. Matrices $2$ and $11$ correspond to the second kind of boundary facets. Matrices $4$ and $6$ correspond to boundary facets of the third kind and matrices $8$ and $9$ correspond to boundary facets of the fourth kind. The remaining matrices correspond to inner faces of the fan. 
\end{rmk}

\begin{thm}\label{thm:convexity}
  The  simplicial fan of pure diagrams in $B_{M,N}^s$ is convex.
\end{thm}

\begin{proof}
  We will use the following observation which allows us to go from local convexity to global convexity: If the simplicial fan is not convex, there will be one boundary facet which supporting hyperplane pass through the interior of a neighboring simplicial cone. We can see this by looking a line segment between two vertices which contains points outside the simplicial fan. If we take a two-dimensional plane through these two points and through generic inner point of the fan, we get a two-dimensional picture where we can find two edges meeting at an inwards angle. The supporting hyperplane of the boundary facet meeting the two-dimensional plane in one of these edges meets the interior of the simplicial cone corresponding to the other edge. 

Thus we will prove that any two boundary facets meet in a convex manner. If any of the two facets are of the first or second kind, described in Proposition~\ref{prop:boundary}, it is clear that any Betti diagram will lie on the correct side of the supporting hyperplane, since this hyperplane is given by the vanishing of a single Betti number. Thus we will assume that the facets are of the third or fourth kind. 

Let $K$ denote the number of pure diagrams in a maximal chain in $B_{M,N}^s$. 
Any two boundary facets meeting along a codimension one face gives us a chain with $K-2$ pure diagrams. If the two missing vertices are on levels differing by more than one, there is a unique way of completing the chain into a maximal chain and the two facets are faces of the same simplicial cone. Hence they meet in a convex way. 

Thus we can assume that the two vertices missing are on adjacent levels and that there a least element $\pi_3$ above these and a greatest element $\pi_0$ below them. Now we can see from the difference in codimension between $\pi_0$ and $\pi_3$ that facets of the third kind described in Proposition~\ref{prop:boundary} cannot meet facets of the fourth kind in this way.

We must have that the two facets are given by removing $\pi_1'$ or $\pi_2''$ from the chains $\pi_0<\pi_1'<\pi_2'<\pi_3$ and $\pi_0<\pi_1''<\pi_2''<\pi_3$ int the following lattice
$$
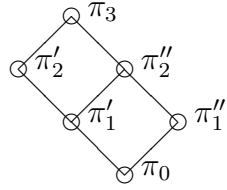
\setlength{\unitlength}{0.7cm}
\begin{picture}(4,3)(0,0)
  \put(2,0){\circle{0.3}$\pi_0$}
  \put(3,1){\circle{0.3}$\pi_1''$}
  \put(1,1){\circle{0.3}$\pi_1'$}
  \put(0,2){\circle{0.3}$\pi_2'$}
  \put(2,2){\circle{0.3}$\pi_2''$}
  \put(1,3){\circle{0.3}$\pi_3$}
  \put(2,0){\line(-1,1){2}}
  \put(2,0){\line(1,1){1}}
  \put(1,1){\line(1,1){1}}
  \put(0,2){\line(1,1){1}}
  \put(3,1){\line(-1,1){2}}
\end{picture}
$$
We need to show that the coefficient of $\pi_1'$ is positive when $\pi_1''$ is expanded into the chain $\pi_0<\pi_1'<\pi_2'<\pi_3$. Note that the codimension cannot differ by more than three between $\pi_0$ and $\pi_3$ and if it differs by three, there is only one chain between $\pi_0$ and $\pi_3$. Thus we can assume that the difference in codimension is zero, one or two. 

In order to do this we use Proposition~\ref{prop:coefficients} which gives us an expression of this coefficient involving only two terms. 

Suppose that $\pi_0$ and $\pi_3$ differ in codimension by two. Then the coefficient is given by the fourth expression of 
Proposition~\ref{prop:coefficients} and the coefficient for $\pi_1'=\pi(d_0,d_1,\dots,d_m)$ when $\pi_1''=\pi(d_0,d_1,\dots,d_{k-1},d_k+1,d_{k+1},\dots,d_m)$ is expanded in the basis containing $\pi_0<\pi_1'<\pi_2'$ is given by 
$$
\sum_{i=0}^n \sum_{d\le d_i(\pi_0)} (-1)^i \prod_{j=0}^{m-1}(d_j-d)\beta_{i,d}(\pi_1'')
$$
which has only two non-zero terms in columns $m$ and $m+1$ and equals
$$
\begin{array}{l}
\displaystyle
(-1)^{m} 
\prod_{j=0}^{m-1} (d_j-d_{m}) (-1)^{m}\frac{d_k-d_{m}}{d_k+1-d_{m}}\cdot\frac{1}{d_{m+1}-d_m}\prod_{j=0}^{m-1}\frac{1}{d_j-d_{m}} \\\displaystyle+ 
(-1)^{m+1} \prod_{j=0}^{m-1} (d_j-d_{m+1}) (-1)^{m+1}\frac{d_k-d_{m+1}}{d_k+1-d_{m+1}}\frac{1}{d_m-d_{m+1}}\prod_{j=0}^{m-1}\frac{1}{d_j-d_{m+1}} \\\displaystyle
=\frac{1}{d_{m+1}-d_m} \left(
\frac{d_k-d_m}{d_k+1-d_m}-\frac{d_k-d_{m+1}}{d_k+1-d_{m+1}}
\right) = \frac{1}{(d_k+1-d_{m+1})(d_k+1-d_m)} >0,
\end{array}
$$
since $d_k+1<d_m<d_{m+1}$. 

If $\pi_0$ and $\pi_3$ differ in codimension by one we get that the coefficient of $\pi_1'$ when $\pi_1''=\pi(d_0,d_1,\dots,d_{m-1}$ is expanded in the basis $\pi_0<\pi_1'<\pi_2'<\pi_3$ is given by 
$$
\sum_{i=0}^n \sum_{d\le d_i(\pi_0)} (-1)^i(d_{k}+2-d_k) \prod_{\underset{j\notin\{k,k+1\}}{j=0}}^{m}(d_j-d)\beta_{i,d}(\pi_1'')
$$
which has only two non-zero terms from columns $k$ and $k+1$ and equals 
$$
\begin{array}{l}
\displaystyle
2(-1)^k
\prod_{\underset{j\notin\{k,k+1\}}{j=0}}^{m}(d_j-d_k)(-1)^k\prod_{\underset{j\ne k}{j=0}}^{m-1}\frac{1}{d_j-d_k} \\\displaystyle
+ 2(-1)^{k+1}\prod_{\underset{j\notin\{k,k+1\}}{j=0}}^{m}(d_j-d_k-1)(-1)^{k+1}\prod_{\underset{j\ne k+1}{j=0}}^{m-1}\frac{1}{d_j-d_k-1} \\=\displaystyle
2\frac{d_m-d_k}{d_k+1-d_k}-2\frac{d_m-d_k-1}{d_k-d_k-1} = 2>0.
\end{array}
$$

The last possibility is the codimension of $\pi_0$ equals the codimension of $\pi_3$ and in this case $\pi_0<\pi_1'<\pi_2'$ corresponds to a facet of the third kind and we get that the coefficient of $\pi_1'=\pi(d_0,d_1,\dots,d_k,d_k+2,d_{k+2},\dots,d_m)$, when $\pi_1''=\pi(d_0,d_1,\dots,d_{\ell-1},d_\ell+1,d_{\ell+1},\dots,d_m)$ is expanded in the basis $\pi_0<\pi_1'<\pi_2'<\pi_3$ is given by 
$$
\sum_{i=0}^n \sum_{d\le d_i(\pi_0)} (-1)^i(d_{k}+2-d_k) \prod_{\underset{j\notin\{k,k+1\}}{j=0}}^{m}(d_j-d)\beta_{i,d}(\pi_1'').
$$
Again there are only two non-zero terms and the coefficient is equal to 
$$
\begin{array}{l}
\displaystyle
2(-1)^k
\prod_{\underset{j\notin\{k,k+1\}}{j=0}}^{m}(d_j-d_k)(-1)^k\frac{d_\ell-d_k}{d_\ell+1-d_k}\prod_{\underset{j\ne k}{j=0}}^{m}\frac{1}{d_j-d_k} \\\displaystyle
+ 2(-1)^{k+1}\prod_{\underset{j\notin\{k,k+1\}}{j=0}}^{m}(d_j-d_k-1)(-1)^{k+1}\frac{d_\ell-d_k-1}{d_\ell+1-d_k-1}\prod_{\underset{j\ne k+1}{j=0}}^{m}\frac{1}{d_j-d_k-1} \\=\displaystyle
2\frac{d_\ell-d_k}{d_\ell+1-d_k}-2\frac{d_\ell-d_k-1}{d_\ell-d_k} = \frac{2}{(d_\ell-d_k)(d_\ell-d_k+1)}>0,
\end{array}
$$
since $d_\ell-d_k$ and $d_\ell-d_k+1$ are both negative or both positive. 
\end{proof}

\section{The expansion of any Betti diagram into sums of pure diagrams}
We know from the work of D.~Eisenbud and F.-O.~Schreyer that the inequalities given by the exterior facets of the cone are valid on all minimal free resolutions, not only on the resolutions of Cohen-Macaulay modules. The inequalities that we have to add because we now look at chains of pure diagrams of different codimensions can be seen to be limits of the inequalities already known. 

As in the previous section, we look at Betti diagrams of graded modules under the restriction that $\beta_{i,j}=0$ unless $M+i\le j\le N+i$, i.e., diagrams in $B_{M,N}$. 

\begin{thm}\label{thm:expansion}
  Any Betti diagram of a finitely generated graded module $P$ of codimension $s$ can be uniquely written as a positive linear combination of pure diagrams in $B_{M,N}^s$, where $M$ is the least degree of any generator of $M$ and $N$ is the regularity of $P$.
\end{thm}

\begin{proof}
  By Proposition~\ref{prop:boundary} we know what are the boundary facets of the simplicial fan given by the pure diagrams in $B_{M,N}$ and by  Proposition~\ref{prop:coefficients} we know how to obtain the inequalities given by the boundary facets. We need to show that any Betti diagram of a graded module satisfies these inequalities. We have two different kinds of inequalities. The first comes from the boundary facets of the third kind and can be written as 
$$
\sum_{i=0}^n \sum_{d=M}^{d_i(\pi_0)} (-1)^i \prod_{\underset{j\notin\{k,k+1\}}{j=0}}^{m}(d_j-d)\beta_{i,d}(P) \ge 0
$$
and the other comes from boundary facets of the fourth kind and can be written as 
$$
\sum_{i=0}^n \sum_{d=M}^{d_i(\pi_0)} (-1)^i \prod_{j=0}^{m-1}(d_j-d)\beta_{i,d}(P)\ge 0.
$$
As we can see from these expressions, they are very similar and the work of D.~Eisenbud and F.-O.~Schreyer shows that the first kind of inequality holds for $m=n$. 

We can see that the second kind of inequality is equal to the first kind of inequality if we increase $N$  and $m$ by one and choose $k=m-1$. Thus it is sufficient to prove that the first kind of inequality always holds for any $m\le n$. 

In order to do this we look at the inequality we get by exchanging $\pi_1=\pi(d_0,d_1,\dots,d_m)$ by $\pi_1^t=\pi(d_0,d_1,\dots,d_m,d_m+1+t,d_m+2+t,\dots,d_m+n-m+t)$, and similarly exchanging $\pi_0$ and $\pi_2$ by $\pi_0^t$ and $\pi_2^t$.
$$
\begin{array}{rl}
\displaystyle
\lim_{t\rightarrow \infty} &\displaystyle
\frac{1}{t^{n-m}} \sum_{i=0}^n \sum_{d=M}^{d_i(\pi_0^t)} (-1)^i \prod_{\underset{j\notin\{k,k+1\}}{j=0}}^{m}(d_j-d)\prod_{j=m+1}^n(d_m+j-m+t-d)\beta_{i,d} \\ &\displaystyle
=
\sum_{i=0}^n \sum_{d=M}^{d_i(\pi_0)} (-1)^i \prod_{\underset{j\notin\{k,k+1\}}{j=0}}^{m}(d_j-d)\beta_{i,d}(P)
\end{array}
$$
and the limit is non-negative since for each integer $t\ge 0$, the expression under the limit in the left hand side is non-negative. 
\end{proof}

\begin{rmk}
  One of the questions raised by D.~Eisenbud and F.-O.~Schreyer was what was the description of the convex cone cut out by all of their inequalities. The answer to this question is that this convex cone in $B_{M,N}$ equals the convex cone of all Betti diagrams in $B_{M,N}$ of graded modules up to multiplication by non-negative rational numbers. This follows from the theorem and the fact that all the inequalities we need when bounding the regularity are given by limits of their inequalities. 

We can also see that the unique way of writing any Betti diagram of a module into a chain of pure diagrams leads to a way of writing the diagram as a linear combination of diagrams of Cohen-Macaulay modules, one of each codimension between the codimension of $M$ and the projective dimension of $M$. As in the Cohen-Macaulay case, we get an algorithm for finding the expansion of a given Betti diagram by subtracting as much as possible of the pure diagram corresponding to the lowest shifts. 
\end{rmk}

\begin{example}
For $M=k[x,y,z]/(x^2,xy,xz^2)$ we get the Betti diagram
$$
\begin{array}{c|cccc}
  &1&3&3&1 \\ \hline
0:&1&-&-&-\\
1:&-&2&1&-\\
2:&-&1&2&1\\
\end{array}
$$
which can be expanded into 
$$
6\cdot\boxed{\begin{array}{cccc}\frac{1}{30}&-&-&-\\-&\frac16&\frac16&-\\-&-&-&\frac{1}{30}\end{array}}
+12\cdot\boxed{\begin{array}{cccc}\frac{1}{40}&-&-&-\\-&\frac{1}{12}&-&-\\-&-&\frac{1}{8}&\frac{1}{15}\end{array}}
+2\cdot\boxed{\begin{array}{cccc}\frac{1}{12}&-&-&-\\-&-&-&-\\-&\frac{1}{3}&\frac{1}{4}&-\end{array}}
+1\cdot\boxed{\begin{array}{cccc}\frac{1}{3}&-&-&-\\-&-&-&-\\-&\frac{1}{3}&-&-\end{array}}
$$
The coefficients can also be obtained by the functionals corresponding to the matrices $5$, $6$, $8$ and $9$ from Example~\ref{ex:functionals}.
\end{example}

We now will go on and prove a generalized version of the Multiplicity Conjecture in terms of the Hilbert series. In order to do this, we first prove that the Hilbert series behaves well with respect to the partial ordering on the normalized pure diagrams. 

\begin{prop}\label{prop:Hilbertseries}
  The Hilbert series is strictly increasing on the partially ordered set of normalized pure diagrams generated in degree zero. 
\end{prop}

\begin{proof}
  First assume that $\bar\pi<\bar\pi'$ is a maximal chain of pure diagrams of codimension $s$. The Hilbert series can be recovered from the Betti diagrams and we can write 
$$
  H(\bar\pi',t)-H(\bar\pi,t) = \frac{1}{(1-t)^n} (S(\bar\pi',t)-S(\bar\pi,t)) 
$$
where $S(\beta,t) = \sum_{i=0}^n (-1)^i \sum_{j} \beta_{i,j}t^j$.
Since the polynomials $S(\bar\pi,t)$ and $S(\bar\pi',t)$ both have constant term $1$, we get that the polynomial $S(\bar\pi',t)-S(\bar\pi,t)$, has only $s+1$ non-zero terms. Since it also satisfies the Herzog-K\"uhl equations, we have a unique solution to this up to a scalar multiple, and $H(\bar\pi,t)-H(\bar\pi',t)$ is $\lambda t^{d_1}$ times the Hilbert series of a pure module. Since we know that all pure diagrams have positive Hilbert series, we only have to check that that this is a non-negative multiple. The sign of $\lambda$ can be obtained by looking at the sign of the term of $S(\bar\pi',t)-S(\bar\pi,t)$ which comes from $S(\bar\pi',t)$ and which is not present in $S(\bar\pi,t)$. This term comes with the sign $(-1)^i$, if it is in column $i$, which proves that $\lambda$ has to by positive. 

We now consider the case where $\bar\pi<\bar\pi'$ is a maximal chain such that the codimension of $\bar\pi$ is $s$ and the codimension of $\bar\pi'$ is $s-1$. 
If we now look at the difference of the Hilbert series, 
$$
H(\bar\pi',t)-H(\bar\pi,t) = \frac{1}{(1-t)^{n}}(S(\bar\pi',t)- S(\bar\pi,t))
$$
where the polynomial $S(\bar\pi',t)-S(\bar\pi,t)$ has zero constant term and $s$ non-zero terms. Since we know that it is divisible by $(1-t)^{s-1}$, there is again a unique possibility, which is $\lambda t^{d_1}$ times the Hilbert series of a module with pure resolution. This time, we can see that the sign of the last term is $-(-1)^{s} = (-1)^{s-1}$, which shows that $\lambda$ is positive. 
\end{proof}

\begin{thm}\label{thm:MultConj} For any finitely generated module $M$ of projective dimension $r$ and codimension $s$ generated in degree $0$, we have that 
$$
\beta_0(M) H(\bar\pi(0,m_1,m_2,\dots,m_r),t) \le H(M,t) \le \beta_0(M)
H(\bar\pi(0,M_1,M_2,\dots,M_s),t),
$$
where $m_1,m_2,\dots,m_r$ are the minimal shifts and $M_1,M_2,\dots,M_s$ are the maximal shifts in a minimal free resolution of $M$. Equality on either side implies that the resolution is pure. 

In particular, the right hand inequality implies the Multiplicity Conjecture, i.e., 
$$
e(M)\le \beta_0(M)\frac{M_1M_2\cdots M_s}{s!}
$$
with equality if and only if $M$ is Cohen-Macaulay with a pure resolution. 
\end{thm}

\begin{proof}
  We can use Theorem~\ref{thm:expansion} to write the Betti diagram of $M$ as a positive linear combination of pure diagrams. Since the Hilbert series by Proposition~\ref{prop:Hilbertseries} is increasing on along the chain, we get the first inequalities of the theorem. 

The multiplicity of $M$ is obtained from the leading coefficient of the Hilbert polynomial. The coefficients of the Hilbert series is eventually equal to the Hilbert polynomial, which shows that the multiplicity is increasing with the Hilbert series as long as the degree of the Hilbert polynomial is the same. Since this is the case for $M$ and $\pi(0,M_1,M_2,\dots,M_s)$, we get the conclusion of the Multiplicity conjecture. 
\end{proof}

\begin{rmk} For a Cohen-Macaulay module $M$, the Hilbert coefficients are positive linear combinations of the entries of the $h$-vector. Since we can reduce $M$ modulo a regular sequence and keep the same Betti diagram we get that the Hilbert coefficients are increasing along chains of normalized pure diagrams. Thus the generalization of the multiplicity conjecture made J.~Herzog and X.~Zheng~\cite{Herzog-Zheng} is a consequence of Theorem~\ref{thm:MultConj}.
\end{rmk}

\end{document}